\documentclass{amsart} 
\usepackage{graphicx}
\usepackage{amssymb, amsmath}
\usepackage{amsfonts}
\usepackage{amssymb}
\usepackage{epsfig}
\usepackage{float}
\usepackage{color}
\newtheorem{theorem}{Theorem}
\newtheorem*{thmA*}{Theorem A}
\newtheorem*{thmB*}{Theorem B}
\newtheorem*{lemA*}{Lemma A}

\newtheorem{definition}{Definition}

\newtheorem{lemma}{Lemma}

\newtheorem{remark}{Remark}

\newtheorem{Winfree Model}{Winfree Model}

\def\be{\begin{equation}}
\def\ee{\end{equation}}

\numberwithin{equation}{section} \numberwithin{theorem}{section}
\numberwithin{figure}{section}
\vspace{1cm}

\usepackage{array,makecell}
\usepackage{graphicx, array, blindtext}

\begin{document}
\bibliographystyle{siam}

\title[ ]
{ On the synchronization theory of Kuramoto oscillators under the effect of inertia  }

\author[C. Hsia]{Chun-Hsiung Hsia}
\address[CH]{Institute of Applied Mathematical Sciences, National Taiwan University, Taipei, Taiwan}
\email{willhsia@math.ntu.edu.tw}
\author[C.-Y. Jung]{Chang-Yeol Jung}
\address[CJ]{Department of Mathematical Sciences, Ulsan National Institute of Science and Technology, Ulsan, Korea}
\email{cjung@unist.ac.kr}
\author[B. Kwon]{Bongsuk Kwon}
\address[BK]{Department of Mathematical Sciences, Ulsan National Institute of Science and Technology, Ulsan, Korea}
\email{bkwon@unist.ac.kr}

\date{\today}



\begin{abstract}
We investigate the synchronized collective behavior of the Kuramoto oscillators with inertia effect. Both the frequency synchronization for nonidentical case and the phase synchronization for identical case are in view. As an application of our general theory, we  show the unconditional frequency synchronization for the three-nonidentical-oscillator case.
\end{abstract}

\maketitle

\section{introduction}
%

The phenomena of synchronization are found in a variety of natural systems. 
The first reported observation of synchronization  dates back to the 17th century; a Dutch scientist, Christiaan Huygens has discovered in 1665 that two pendulum clocks hanging on the wall have always ended up swinging in exactly the opposite direction from each other.  
Since then, various synchronization phenomena have been observed. 
These include circadian rhythms, electrical generators, Josephson junction arrays, intestinal muscles, menstrual cycles, and fireflies.
Yet, the underlying mechanism of synchronization has remained a mystery.
Among a number of mathematical models, the ones proposed by Kuramoto \cite{Ku75} and Winfree \cite{Win67}, respectively have received  considerable attention. 
  
The Kuramoto model reads
\begin{equation}
\label{k1}
 \dot{\theta}_i = \omega_i + \frac{K}{N}\sum_{j=1}^{N} \sin(\theta_j - \theta_i), \,\, t>0, \,\, i=1,2, \cdots, N,
\end{equation}
where $\theta_i = \theta_i(t) \in \mathbb R$ is the phase of $i-$th oscillator with the intrinsic natural frequency $\omega_i$, $K$ is the coupling strength constant, $N$ is the number of the oscillators, and $\dot{\theta}_i$ denotes the first derivative of $\theta_i$. 
System \eqref{k1}  is first introduced by Yoshiki Kuramoto in \cite{Ku75, Ku84} to describe synchronization phenomena observed in the systems of chemical and biological oscillators.
This model makes assumptions  that 
(i) the oscillators are all-to-all, weakly coupled,
(ii) the interactions  between two oscillators  depends sinusoidally on the phase difference. 
 There have been extensive studies for the Kuramoto model, and we refer the interested readers to \cite{Er85,Er91,HW93} for motivation and physicality of the model. 
 Later on, more physical effects were incorporated in the Kuramoto model. Our interest in this article is to develop a general theory for the following Kuramoto model under the effect of inertia: 
 \begin{equation}
\label{kuramoto2nd}
m \ddot{\theta}_i + \dot{\theta}_i = \omega_i + \frac{K}{N}
\sum_{j=1}^{N}\sin(\theta_j - \theta_i),  \,\, t > 0, \,\, i
=1,2,3,\cdots N.
\end{equation}
  This model was first introduced by Ermentrout \cite{Er91} to explain the synchronization of fireflies of the Pteroptyx malaccae species. Although this model has received a lot of attention for years, the development of the mathematical analysis on the synchronization theory for \eqref{kuramoto2nd} is still in the early stage.

For our convenience of further discussion, we shall introduce several notations and functions here.
Let $$\Theta(t) := ( \theta_1(t), \theta_2(t), \theta_3(t), \cdots, \theta_N(t))$$  be a vector-valued function $\Theta: [0,\infty) \to \mathbb{R}^N$, and $$\Omega:=(\omega_1, \omega_2, \omega_3, \cdots, \omega_N) $$ be a constant vector in $\mathbb{R}^N$. For any positive integer $M\le N$, $\Theta_M(t)$ denotes the vector consisting of the first $M$ components of $\Theta(t),$ i.e., $$\Theta_M(t) := (\theta_1(t), \cdots,\theta_M(t) ),$$ and  $\Omega_M$ denotes, similarly, the first $M$ components of the natural frequency vector $\Omega.$ For any finite dimensional vector $X=(x_1, x_2, x_3, \cdots, x_N)$, we  define 
$$D(X) := \max_{i\ne j} (x_i -x_j).$$  
Let 
$\omega$ be the mean of $\Omega$, i.e., 
\begin{equation*}
\label{meanomega}
 \omega := \frac{1}{N}\sum_{i=1}^N \omega_i .
\end{equation*}


\begin{definition}\label{def1}
For both systems \eqref{k1} and \eqref{kuramoto2nd}, we say the Kuramoto oscillator ensemble $\Theta(t)$  achieves \emph{the complete-frequency synchronization asymptotically} if
$$\lim_{t\to\infty} \dot \theta_i(t) = \omega  \ \ \text{ for all } 1\le i \le N.$$
\end{definition}

\begin{remark} 
This means that the differences of the velocities of any two different oscillators tend to zero as time tends to infinity. In fact each velocity of the $i$-th oscillator tends to the same constant velocity, i.e. the oscillators will have the same frequencies in the limit of $t\to\infty$. 
\end{remark}
\begin{definition}\label{def2}
For both systems \eqref{k1} and \eqref{kuramoto2nd}, we say the Kuramoto oscillator ensemble $\Theta(t)$  achieves \emph{the complete phase synchronization asymptotically} if there exist integers  $k_{ij}$ such that
$$\lim_{t\to\infty} | \theta_i(t) - \theta_j(t) - 2k_{ij} \pi| = 0  \ \ \text{ for all } 1\le i< j \le N.$$
\end{definition}

\begin{remark}
It is easy to see that an obvious necessary condition for the systems \eqref{k1} or \eqref{kuramoto2nd} to exhibit the complete phase synchronization is that
\begin{equation}
\label{identical}
\omega_i = \omega   \ \ \text{ for all }  i = 1,2,3, \cdots, N.
\end{equation}
When \eqref{identical} holds, 
the oscillators are said to be identical.
\end{remark}

\begin{definition}
The diameter function $D(\Theta(t))$ is defined by 
\begin{equation}\label{dD}
D(\Theta(t) ):= \max_{1\le i,j\le N} ( \theta_i(t) - \theta_j(t) ).  
\end{equation}
\end{definition}
\begin{definition}
For $i \ne j$, we say the oscillators $\theta_i$ and $\theta_j$  collide at time $t_*>0$ if $$\theta_i(t_*) = \theta_j(t_*).$$ \end{definition}
\begin{remark}
\label{rm2}
For any initial data $\Theta(0)\in\mathbb{R}^N$ and $\dot{\Theta}(0) \in \mathbb{R}^N$, the solution $\Theta(t)$ to \eqref{kuramoto2nd}  is analytic for all $t\ge0$. It is easy to see that the diameter function $D(\Theta(t))$ is a continuous function. In general, $D(\Theta(t))$ may not be  $C^2$ nor $C^1$ for all $t >0$ because of the onsets of collisions between the oscillators. However, we can still  estimate  the rate of the change  of the diameter function by considering the first derivatives of all the representations of the diameter functions. For example,
when the collision occurs at a certain moment $t=t_0$, there are more than one representation of the diameter functions, i.e., 
 $$D(\Theta(t_0))=\theta_{i_k}(t_0)-\theta_{j_k}(t_0), \ \ k=1,2,\cdots,m,$$
 for some $m\ge2$. If we have the estimates 
\begin{equation*}
\dot{\theta}_{i_k}(t_0)-\dot{\theta}_{j_k}(t_0) < \delta <0,
\end{equation*}
for $k=1,2,\cdots,m$, we can conclude that $D(\Theta(t))$ is strictly decreasing in a small neighborhood of $t=t_0$. We shall use this observation frequently in the proof of our main results.
\end{remark}

The first main result of this article is the following  complete frequency synchronization theorem.
\begin{theorem}
\label{thm2}
Let $M>\frac{N}{2}$ be a positive integer and   $0<\beta<\alpha$ such that $$2\beta+ \alpha < \pi$$ and
$$
   \frac{M}{N}\sin(\frac{\alpha}{2} - \frac{\beta}{4}) \cos(\frac{\alpha}{2}+\frac{5}{8}\beta)-(1-\frac{M}{N})\cos(\frac{\alpha}{2} - \frac{\beta}{8})>0
.$$
Let $\mu$ and $\lambda$ be such that 
\begin{equation}
\label{mu2}
0< \mu \le \frac{M}{N}\sin(\frac{\alpha}{2} - \frac{\beta}{4}) \cos(\frac{\alpha}{2}+\frac{5}{8}\beta)-(1-\frac{M}{N})\cos(\frac{\alpha}{2} - \frac{\beta}{8}). 
\end{equation}
and $\lambda > \mu +2$. 
 Assume that 
\begin{align}
\label{mk2}& mK \le \beta \Big/ \Big(4(\lambda + \mu +2) \ln (\frac{ \lambda+ 2\mu +2 } { \mu } )\Big)\,\, \text{ and }\\
\label{domega2}& D(\Omega) < \mu K.
\end{align}
Then, the solution of \eqref{kuramoto2nd} with the initial conditions 
\begin{align}
\label{alpha2}&D(\Theta_{M}(0)) \le \pi -  \alpha - \beta  \text{ and }\\
\label{lambda2}& D(\dot{\Theta}_{M}(0)) < \lambda K
\end{align}
achieves the complete frequency synchronization asymptotically, i.e., 
\begin{equation}
\lim_{t\to\infty} \dot\theta_i(t) = \omega \ \ \text{ for all } 1\le i \le N.
\end{equation}
\end{theorem}

For the identical case, i.e., $\omega_i=\omega$ for all $i=1,2,\cdots,N$, we have the \emph{phase synchronization} result as follows.
\begin{theorem}
\label{thm3}
Let  $0<\beta < \alpha < \pi$ such that 
$$2\beta + \alpha < \pi$$ and
\begin{equation}
\label{pmu2}
\mu = \sin(\frac{\alpha}{2} - \frac{\beta}{4})\cos(\frac{\alpha}{2}+\frac{5}{8}\beta).
\end{equation}
Let $\lambda > \mu +2$ be any fixed positive number. 
 Assume that 
\begin{align}
\label{pmk2}& mK \le \beta \Big/ \Big(4(\lambda + \mu +2) \ln (\frac{ \lambda+ 2\mu +2 } { \mu } )\Big)\,\, \text{ and }\\
\label{pdomega2}& D(\Omega) =0.
\end{align}
Then, the solution of \eqref{kuramoto2nd} with the initial conditions 
\begin{align}
\label{palpha2}&D(\Theta(0)) \le \pi -  \alpha - \beta  \text{ and }\\
\label{plambda2}& D(\dot{\Theta}(0)) < \lambda K
\end{align}
achieves the complete phase synchronization asymptotically, i.e., 
$$\lim_{t\to\infty} \left( \theta_i(t) - \theta_j(t) \right)  =0 \ \ \text{ for all } 1\le i,j\le N.$$
\end{theorem}

The diameter function plays an important role  in our synchronization analysis.
First, it is obvious that $\lim_{t \to \infty} D(\Theta(t))=0$ implies the phase synchronization.
On the other hand, as we shall see in Lemma~\ref{lemdiameter} in Section 2, the uniform boundedness of $D(\Theta(t))$, i.e.,
\begin{equation}
\label{diameter111}
\sup_{t \ge 0} D(\Theta(t)) < \infty,
\end{equation}
implies the  complete frequency synchronization. 
As mentioned in Remark~\ref{rm2}, the diameter function $D(\Theta(t))$ for \eqref{kuramoto2nd}, in general, may not be $C^2$ nor $C^1$ for all $t \ge 0$. Hence, it is not appropriate to consider the second order differential inequality for $D(\Theta(t))$ and to apply the Grownwall type inequalities  to carry out the asymptotic analysis for the system \eqref{kuramoto2nd}.  On the other hand, to study the rate of change of $D(\Theta(t))$, one should take the onsets of collisions into consideration carefully. Due to the onsets of collision, to estimate the magnitude of the diameter function for all $t\ge0$, one requires more information than just certain constraint on $D(\Theta(0))$ and $\dot{D}(\Theta(0)).$ Here $\dot{D}(\Theta(t))$ is defined as
\begin{equation}
\label{ddd}
\dot{D}(\Theta(t)) = \dot{\theta}_i(t)-\dot{\theta}_j(t),
\end{equation}
 where $\theta_i(t)-\theta_j(t)$ is a representation of $D(\Theta(t))$ that maximizes the right hand side of  \eqref{ddd}.  
In literatures, there are some invalid arguments derived from the aforementioned inappropriate scenario. 
In general, 
it is impossible to get the estimate of $D(\Theta(t))$ of the system \eqref{kuramoto2nd} solely by the initial data: $D(\Theta(0))$ and $\dot{D}(\Theta(0))$. The problem arises when the collisions occur. 
On the occasion of collisions,  $\dot{D}(\Theta(t))$ may get suddenly large  and drive  $D(\Theta(t))$ much larger than it was.  
An easy example is given as follows:
Let $N=4$,  $\omega_i=0$ for $i=1,2,3,4,$  and set
\begin{align}
\label{initial}& \theta_1(0)=0, &&\theta_2(0)=\epsilon_1, &&&\theta_3(0)=\epsilon_2, &&&&\theta_4(0) = \epsilon_3\\
\label{initialvelocity}&\dot{\theta}_1(0)=0, &&\dot{\theta}_2(0)=-a, &&&\dot{\theta}_3(0)=a, &&&&\dot{\theta}_4(0)=0,
\end{align}
where  $0<\epsilon_1<\epsilon_2<\epsilon_3<\frac{\pi}{6}$ and $a>0$.
One can readily check that 
\begin{align*}
& D(\Theta(0))=\theta_4(0)-\theta_1(0)=\epsilon_3,\\
 &\dot{D}(\Theta(0))=\dot{\theta}_4(0)-\dot{\theta}_1(0)=0
  \end{align*}
  and initially (for $t>0$ small), $D(\Theta(t))$ satisfies the differential inequality 
$$m\ddot{D}(\Theta(t))+\dot{D}(\Theta(t)) \le 0.$$
However,  it is impossible to conclude that $D(\Theta(t))$ is bounded by some prescribed number depending only on $D(\Theta(0))$ and $\dot{D}(\Theta(0))$. In fact,
similarly as in the proof of  Lemma~\ref{lemma1}, we can derive that
\begin{equation*}
\dot{\theta}_3(t) - \dot{\theta}_2(t) \ge 2a e^{-\frac{1}{m}t} -2K,
\end{equation*}
and 
\begin{equation*}
\theta_3(t) - \theta_2(t) \ge  \epsilon_2 - \epsilon_1 + 2a (1- \frac{1}{m} e^{-\frac{1}{m}t}) - 2Kt.
\end{equation*}
Choose $t=T>0$ such that $\frac{1}{m}e^{-\frac{1}{m}T} < \frac{1}{2} $. We then see
\begin{equation*}
D(\Theta(T))\ge\theta_3(T) - \theta_2(T) \ge  \epsilon_2 - \epsilon_1 + a  - 2KT.
\end{equation*}
 The right hand side of the above inequality can be arbitrarily large by choosing $a>0$ large. 

 
 In this article, we carry out the asymptotic analysis for $D(\Theta(t))$ in a   comprehensive way. To estimate the rate of the change of $D(\Theta(t))$, we treat \eqref{kuramoto2nd} as  first order differential equations of $\dot{\Theta}(t)$, and take advantage of the boundedness of the nonlinear terms to obtain the following estimate.
\begin{lemma}
\label{lemma1}
For any $i,j \in \{1,2,3,\cdots,N \}$, the solution $\Theta(t)$ of \eqref{kuramoto2nd} satisfies
\begin{equation}
\label{difference}
| \dot{\theta}_i(t) -  \dot{\theta}_j(t) | \le | \dot{\theta}_i(0) - \dot{\theta}_j(0)|  e^{-\frac{1}{m}
t}+ (| \omega_i - \omega_j |+2K)(1 - e^{-\frac{1}{m} t})  \\ 
\end{equation}
for all $t \ge 0.$
\end{lemma}

Based on Lemma~\ref{lemma1}, we are able to get some nonlocal estimates to establish the  following sector trapping lemma, which plays a key role in the proof of our main theorems.

\begin{lemma}
\label{lem3}

Let $M>\frac{N}{2}$ be a positive integer and   $0<\beta<\alpha$ such that $$2\beta+ \alpha < \pi$$ and 
$$ \frac{M}{N}\sin(\frac{\alpha}{2} - \frac{\beta}{4}) \cos(\frac{\alpha}{2}+\frac{5}{8}\beta)-(1-\frac{M}{N})\cos(\frac{\alpha}{2} - \frac{\beta}{8})>0.$$
Let $\mu$ and $\lambda$ be the numbers satisfying
\begin{equation}
\label{mu3}
0< \mu \le \frac{M}{N}\sin(\frac{\alpha}{2} - \frac{\beta}{4}) \cos(\frac{\alpha}{2}+\frac{5}{8}\beta)-(1-\frac{M}{N})\cos(\frac{\alpha}{2} - \frac{\beta}{8})
\end{equation}
and $\lambda > \mu +2$. 
Assume that \eqref{mk2}-\eqref{domega2} hold and  $\Theta(t)$ is a solution of \eqref{kuramoto2nd} that satisfies \eqref{alpha2} - \eqref{lambda2}. Then, 
$$
\sup_{t \ge 0} D(\Theta_M(t)) \le \pi - \alpha,
$$
and 
\begin{equation}
\label{alpha111}
D(\Theta_M(t))\le  \pi- \alpha-\beta
\end{equation}
for all 
$$
t > \tau + \frac{\beta}{K(\mu-(\lambda + 2 \mu +2)e^{-\frac{1}{m} \tau})},
$$
where $\tau$ is given by
\begin{equation}
\label{tau2f}
\tau := \frac{\beta}{ 4(\lambda + \mu +2) K}.
\end{equation}

\end{lemma}
Our strategy to prove Theorem~\ref{thm2}  is to show that under the assumptions of Theorem~\ref{thm2}, by Lemma~\ref{lem3}, the inequality 
\eqref{alpha111} holds. 
We can then apply Lemma~\ref{lemma8} to obtain \eqref{diameter111}. 
Finally, by applying Lemma~\ref{lemdiameter}, we obtain the conclusion of  Theorem~\ref{thm2}.  
The proof of Theorem~\ref{thm3} is based on Theorem~\ref{thm2} and Lemma~\ref{lem3}.

To state the next theorem, we define 
 \begin{equation}
\begin{split}
\label{beta3}
\lambda_3 & : = \max_{ \beta \in(0, \frac{2}{5}\pi] } \left( \frac{9}{10}\sin \frac{\beta}{16} \cos \frac{11 \beta}{16} \right)  \\
& = \frac{9}{10}\sin \frac{\alpha_3}{16} \cos \frac{11 \alpha_3}{16}
\end{split}
\end{equation}
where $\alpha_3 = 2\pi / 5$.
\begin{theorem}
\label{theorem3}
 For $N=3$, suppose 
\begin{align}
\label{mk3}& mK \le \alpha_3 \Big/ \Big(4( \sin \frac{\alpha_3}{16} \cos \frac{11 \alpha_3}{16} +2) \ln (\frac{4\sin \frac{\alpha_3}{16} \cos \frac{11 \alpha_3}{16}+6} { \sin \frac{\alpha_3}{16} \cos \frac{11\alpha_3}{16} } )\Big),\\
\label{domega3}& D(\Omega) < \lambda_3K.
\end{align}
 Then for any initial phase $\Theta(0)$ and velocity $\dot\Theta(0)$, the initial value problem for \eqref{kuramoto1} 
 admits a unique solution $\Theta\in C^\infty([0,\infty))$ satisfying 
\begin{equation}
\lim_{t\to\infty} \dot\theta_i(t) = \omega \ \ \text{ for all } 1\le i \le N.
\end{equation}
I.e., system \eqref{kuramoto1} exhibits the unconditional (independent of initial conditions) complete frequency  synchronization.
\end{theorem} 

We remark that Theorem \ref{theorem3} does not require any condition on $(\Theta(0),\dot\Theta(0))$. 
This is referred to as the \emph{unconditional} synchronization.
  To the best of our knowledge, this is the first result of the \emph{unconditional} complete-frequency synchronization for the second order Kuramoto model of  \emph{non-identical} three oscillators, i.e.,  \eqref{kuramoto2nd} with $N=3$. 
 It is naturally conjectured, by numerical experiments, that the Kuramoto oscillators  will frequency-synchronize \emph{unconditionally} when $D(\Omega)/K$ is small. This has been an open question (even for the first order model  $m=0$), and up to date, no satisfactory theory has been established.

 In general, it is numerically demonstrated that  synchronization happens only when $D(\Omega)/K$ is small, i.e., in the case of either 
small frequency detuning
or large coupling.
In light of this, the coupling strength condition \eqref{domega3} in Theorem~\ref{theorem3} is a natural condition. 

%

Mathematically, the second order Kuramoto model is more difficult than the first order model to analyze. As we mentioned earlier, one difficulty arises from  collision of oscillators. We note that the second order model may have multiple collisions, similarly as the  first order model of nonidentical oscillators. 
Due to the collision, some specific techniques including the ordinary differential inequality for the diameter function that have been successfully used for the first order models, do not seem directly applicable to the second order models. However, the method we employ here works out even in the presence of collisions, which are mostly the cases.

\section{Complete  synchronization}

The non-identical second order Kuramoto system with uniform inertia $m>0$ reads
\begin{equation}
\label{kuramoto1}
m \ddot{\theta}_i + \dot{\theta}_i = \omega_i + \frac{K}{N}
\sum_{j=1}^{N}\sin(\theta_j - \theta_i),  \,\, t > 0, \,\, i
=1,2,3,\cdots N.
\end{equation}

Let 
$\omega$ be the mean of $\Omega=\{\omega_i\}_{i=1}^N$ and $\theta(t)$ be the mean of $\Theta(t)=\{ \theta_i(t)\}_{i=1}^N$, i.e., 
\begin{align*}
\label{meanomega}
 & \omega := \frac{1}{N}\sum_{i=1}^N \omega_i \qquad \text{ and }\\
 & \theta(t) := \frac{1}{N}\sum_{i=1}^N \theta_i(t).
\end{align*}
Then it is easy to see that $\theta(t)$ satisfies 
\begin{equation}
m\ddot{\theta} + \dot{\theta} = \omega.
\end{equation}
Integrating it over $(0,t)$, 
we see that
\begin{equation}
\dot{\theta}(t) = \dot{\theta}(0)e^{-\frac{1}{m}t} + \omega(1 - e^{-\frac{1}{m}t}),
\end{equation}
which implies that 
\begin{equation}
\sup_{t \ge 0} | \dot{\theta}(t) | \le |\dot{\theta}(0) | + |\omega|, 
\end{equation}
and
\begin{equation}
\lim_{t \to \infty} \dot{\theta}(t) = \omega.
\end{equation}
 We  observe that by the change of variables:
 \begin{equation}
 \begin{aligned}
 & \bar{\theta}_i(t) = \theta_i(t) - \omega t, \\
 & \bar{\omega}_i = \omega_i -\omega,
 \end{aligned}
 \end{equation}
 for $ i = 1,2,3, \cdots, N,$ system \eqref{kuramoto1}  can be rewritten as 
\begin{equation}
\label{kuramoto2}
m \ddot{\bar{\theta}}_i + \dot{\bar{\theta}}_i = \bar{\omega}_i + \frac{K}{N}
\sum_{j=1}^{N}\sin(\bar{\theta}_j - \bar{\theta}_i),  \,\, t > 0, \,\, i
=1,2,3,\cdots N.
\end{equation}
By a direct calculation, we have
 \begin{align}
\label{ddthe}\dot{ \bar{\theta}}_i(t) - \dot{\bar{\theta}}_j(t) = \dot{\theta}_i(t) - \dot{\theta}_j(t),\\
\label{dthe}\bar{\theta}_i(t) - \bar{\theta}_j(t) = \theta_i(t) - \theta_j(t),
 \end{align}
and
\begin{equation}
\label{barmean}
\sum_{i=1}^N \bar{\omega}_i = 0.
\end{equation}
The relations \eqref{ddthe} and \eqref{dthe} reveal that the  problems to show  (frequency/phase)  synchronization   of the two systems \eqref{kuramoto1} and \eqref{kuramoto2} are equivalent. By \eqref{barmean}, we see that when considering the synchronization problem, without loss of generality, we may always assume the mean zero condition on frequency, i.e.,
\begin{equation}
\label{meanzerof}
\omega= \frac{1}{N} \sum_{i=1}^N \omega_i =0.
\end{equation}

 \begin{lemma}
 \label{meanphase}
 Assume the mean zero condition on frequency \eqref{meanzerof} holds,
 and $\Theta(t)$ is the solution of   system \eqref{kuramoto1} with the initial conditions $\Theta(0)$ and $\dot{\Theta}(0)$. Then, $\theta(t)$ is uniformly bounded for all $t\ge0$. More precisely, the following inequality holds 
 \begin{equation}
 \label{center1}
 \sup_{t \ge 0}|\theta(t)| \le |\theta(0)| + m |\dot{\theta}(0)|.
 \end{equation}
 \end{lemma}
 \begin{proof}
 Solving the mean equation of \eqref{kuramoto1} for $\theta(t)$ with initial conditions $\theta(0)$ and $\dot{\theta}(0)$ gives
 \begin{equation}
 \theta(t) = \theta(0) +m \dot{\theta}(0)(1-e^{-\frac{1}{m}t}),
 \end{equation}
 which leads to the desired inequality.
 \end{proof}
 
\noindent We now prove Lemma~\ref{lemma1}.

\noindent \emph{The proof of Lemma~\ref{lemma1}}

For $i,j \in \{1,2,3,\cdots,N \},$ by \eqref{kuramoto1}, we see that
\begin{equation}
\label{difference0}
\begin{aligned}
\dot{\theta}_i(t) - & \dot{\theta}_j(t) =(\dot{\theta}_i(0) - \dot{\theta}_j(0)) e^{-\frac{1}{m}
t} + (\omega_i - \omega_j)(1-
e^{-\frac{1}{m}t}) \\
& + \int_0^t \frac{K}{mN} e^{-\frac{1}{m}(t -s)} \sum_{k=1}^N
\Big( \sin(\theta_k(s) - \theta_i(s))-  \sin(\theta_k(s) -
\theta_j(s))  \Big) ds.
\end{aligned}
\end{equation}
We notice that, in any case, we have
\begin{align*}
\Big| \int_{0}^{t} \frac{K}{mN} & e^{-\frac{1}{m}(t -s)} \sum_{k=1}^N
\Big( \sin(\theta_k(s) - \theta_i(s))-  \sin(\theta_k(s) -
\theta_j(s))  \Big) ds \Big| \\
& \le 2K \int_0^{t} \frac{1}{m} e^{-\frac{1}{m} (t-s)} ds  \\
&\le 2K(1 - e^{-\frac{1}{m} t}) .
\end{align*}
Hence,  for  any   $i, j \in \{1,2,3, \cdots, N \}$ , we have 
\begin{equation}
\label{difference1}
| \dot{\theta}_i(t) -  \dot{\theta}_j(t) | \le | \dot{\theta}_i(0) - \dot{\theta}_j(0)|  e^{-\frac{1}{m}
t}+ (| \omega_i - \omega_j |+2K)(1 - e^{-\frac{1}{m} t})  \\ 
\end{equation}
for all $t \ge 0.$
\qed

Similarly as in the proof of Lemma~\ref{lemma1}, we can derive the following lemma.
\begin{lemma}
\label{lemmathetai}
For any $i \in \{1,2,3,\cdots,N \}$, the solution $\Theta(t)$ of \eqref{kuramoto1} satisfies
\begin{equation}
\label{thetadoti}
| \dot{\theta}_i(t)  | \le | \dot{\theta}_i(0)|  e^{-\frac{1}{m}
t}+ (| \omega_i |+K)(1 - e^{-\frac{1}{m} t})  \\ 
\end{equation}
for all $t \ge 0.$
\end{lemma}
 Plugging \eqref{thetadoti} into \eqref{kuramoto1}, we obtain the following lemma.
 \begin{lemma}
\label{lemmathetai2}
For any $i \in \{1,2,3,\cdots,N \}$, the solution of \eqref{kuramoto1} satisfies
\begin{equation}
\label{thetadoti2}
m| \ddot{\theta}_i(t)  | \le | \dot{\theta}_i(0)|  e^{-\frac{1}{m}
t}+ (| \omega_i |+K)(1 - e^{-\frac{1}{m} t}) + |\omega_i|+K  \\ 
\end{equation}
for all $t \ge 0.$
\end{lemma}

 One of the key features of the present work is, by using an appropriate Lyapunov functional and energy method, to establish a simple sufficient condition for the frequency synchronization, that is the uniform boundedness of diameter function, i.e., $\sup_{t>0} D(\Theta(t) )<\infty$.
 This concept has been successfully applied to the first order Kuramoto models in literature. For the application to the second order Kuramoto models, it has not yet been so well understood. We shall give a complete account of this idea. Namely, we shall build up Lemma~\ref{lemdiameter}, 
 which gives a handy criterion for the asymptotic complete frequency synchronization for the second order Kuramoto system \eqref{kuramoto1}. However, to bridge the gap between the assumptions of Theorem~\ref{thm2} and Lemma~\ref{lemdiameter}, we need a delicate idea which is illustrated by Lemma~\ref{lemma8}.

We define a functional measuring a total frequency energy of the oscillators in the square summation sense:
 \begin{align*}
 E_N(t) := 
 \sum_{ i =1}^{ N} | \dot{\theta}_{i}(t)  |^2.
\end{align*}
Making use of this, and by multiplying \eqref{kuramoto1} by  $\dot{\theta}_i$, summing over  $i=1,\cdots,N$, and integrating it over $[0,t]$, we obtain
\begin{equation}\label{m-km}
\frac{m}{2} E_N(t) + \int_0^t E_N (s) ds =  \frac{m}{2} E_N(0) + \int_0^t \Lambda_N(s) ds + \int_0^t H_N(s) ds,
\end{equation}
\begin{align}
&\Lambda_N(t) : = \sum_{i=1}^N\omega_i \dot{\theta}_i(t), \\ 
&H_N(t) := \frac{K}{N}\sum_{i,j=1}^N \sin(\theta_j -\theta_i)\dot{\theta}_i(t).
\end{align}
Here, by using symmetry of $H_N$, we find that 
\begin{equation}\label{hh}
H_N(t)  =    \frac{K}{N} \frac{d}{dt} \sum_{1\le i<j\le N} \cos(\theta_j-\theta_i)(t).
\end{equation}
In order  to prove the frequency  synchronization, we find out that it suffices to establish the uniform bounds for the right-hand side of \eqref{m-km}, i.e., 
\begin{equation}\label{bdds}
\sup_{t>0} \big|\int_0^t \Lambda_N (s)ds \big| <\infty \text{ and  } \sup_{t>0} \big| \int_0^t H_N(s)ds  \big|<\infty.
\end{equation}
It is obvious that
\begin{equation}
\label{hbounded}
\big| \int_0^tH_N(s)ds \big|= \big| \frac{K}{N}\sum_{i < j}\cos(\theta_j -\theta_i)\Big|_{0}^t \big| \le (N-1)K
\end{equation}
for all $t \ge 0.$
In fact, if \eqref{bdds} holds, the following limit exists and it is finite:
 $$\sup_{t>0} \int_0^t E_N(s) ds = \int_0^\infty E_N(s) ds <\infty.$$ 
 By Lemma~\ref{cvl} in the appendix, this together with the uniform continuity of $E_N(t)$ on $[0,\infty)$ implies that $E_N(t)\to 0$ as $t\to\infty$. This immediately deduces $\dot{\Theta}(t) \to 0$ as $t\to\infty$, which implies the complete frequency synchronization. 
 We notice that the uniform continuity of $E_N(t)$ is a direct result of Lemma~\ref{lemmathetai} and Lemma~\ref{lemmathetai2}.
 Namely, by the above argument, we obtain   the following   criterion for the complete frequency  synchronization for the second order Kuramoto system \eqref{kuramoto1}.
  \begin{lemma}
\label{thmnonidentical}
If $\Theta(t)$ is a solution of \eqref{kuramoto1} under the condition \eqref{meanzerof} such that    
\begin{equation}\label{bdd1}
\sup_{t\ge0} \big|\int_0^t \Lambda_N(s)ds \big| < M,
\end{equation}
for some $M>0$,
then 
\begin{equation*}
\lim_{t \to \infty} \dot\Theta(t) =0.
\end{equation*}
\end{lemma}

However, in general, it is not easy to show the inequality \eqref{bdd1}.   Instead of proving the inequality \eqref{bdd1} directly, we shall show that the diameter function $D(\Theta(t))$ is uniformly bounded, i.e.,
\begin{equation}
\label{diameterbound}
\sup_{t \ge 0} D(\Theta(t)) < \infty.
\end{equation}
We notice that under the mean zero condition on the frequency \eqref{meanzerof}, the inequality  \eqref{center1} of Lemma~\ref{meanphase} holds. A simple observation reveals that  \eqref{diameterbound} together with \eqref{center1} implies
\eqref{bdd1}. We therefore obtain the following useful criterion for the frequency synchronization.
\begin{lemma}
\label{lemdiameter}
Under the assumption \eqref{meanzerof}, if $\Theta(t)$ is a solution of \eqref{kuramoto1}  such that    \eqref{diameterbound} holds,
then 
\begin{equation*}
\lim_{t \to \infty} \dot\Theta(t) =0.
\end{equation*}
\end{lemma}

We now prove Lemma~\ref{lem3}.

\noindent \emph{Proof of Lemma~\ref{lem3}.}

\noindent{\sc Step 1.}
Applying Lemma~\ref{lemma1} and taking \eqref{domega2} and \eqref{lambda2} into consideration, we see
for any $m, n \in \{1,2,3,\cdots,M \}$, that
\begin{align}
\label{max2} | \dot{\theta}_m(t) - \dot{\theta}_n(t)| &\le  | \dot{\theta}_m(0) - \dot{\theta}_n(0)|  
+ (| \omega_m - \omega_n |+2K)\\
\nonumber &\le    (\lambda + \mu +2) K
\end{align}
for all $t \ge 0.$
Set
\begin{equation}
\label{tau2}
\tau := \frac{\beta}{ 4(\lambda + \mu +2) K},
\end{equation}
where $\beta$ is the number defined in Lemma~\ref{lem3}. 
By \eqref{alpha2} and \eqref{tau2}, it follows from \eqref{max2} that  we have
\begin{equation}
D(\Theta_M(t)) < \pi - \alpha  \quad \text{  for  all } 0\le t \le 4 \tau.
\end{equation}

\noindent{\sc  Step 2.}
For any moment $t \ge \tau$ such that 
$$ \pi-\alpha-\beta \le D(\Theta_{M}(t))  \le \pi - \alpha,$$
let $\theta_i(t) - \theta_j(t)$ be one of the choices that realizes 
$D(\Theta_M(t)),$ i.e., 
\begin{equation}
\label{diameter2}
\theta_i(t) - \theta_j(t) = D(\Theta_{M}(t)).
\end{equation}
By \eqref{max2}--\eqref{diameter2}, we see that for $t - \tau \le s \le t,$  
$$ \theta_i(t) - \theta_j(t) - \frac{\beta}{4} \le \theta_i(s) - \theta_j(s) \le \theta_i(t) - \theta_j(t) + \frac{ \beta}{4},$$
and
\begin{align*}
& \Big| \theta_k(s) - \frac{\theta_i(s)+ \theta_j(s)}{2} \Big|   \le   \Big| \theta_k(t) -\frac{1}{2}( \theta_i(t) + \theta_j(t)) \Big| + \frac{ \beta}{4} \le  \frac{1}{2}( \theta_i(t) - \theta_j(t))+\frac{\beta}{4},
\end{align*}
for $ k \in \{1,2, \cdots, M\}$.
Hence, one has 
\begin{align}
 \label{ss1} \sin(\frac{\theta_i(s) - \theta_j(s)}{2})&\ge \sin( \frac{1}{2}( \theta_i(t) - \theta_j(t)) - \frac{\beta}{8}) \ge \sin(\frac{\pi}{2}- \frac{\alpha}{2} -\frac{\beta}{2}- \frac{\beta}{8})\\
 \nonumber& \ge \cos(\frac{\alpha}{2}+\frac{5}{8}\beta),\\
\label{ss15} \sin(\frac{\theta_i(s) - \theta_j(s)}{2})&\le \sin( \frac{1}{2}( \theta_i(t) - \theta_j(t)) + \frac{\beta}{8}) \le \sin( \frac{\pi}{2}-\frac{\alpha}{2} + \frac{\beta}{8})\\
\nonumber & \le \cos(\frac{\alpha}{2}-\frac{\beta}{8}),
  \end{align}
 and
 \begin{align}
 \label{ss2}&   \cos (\theta_k(s) - \frac{\theta_i(s)+ \theta_j(s)}{2} ) \ge \cos( \frac{1}{2}( \theta_i(t) - \theta_j(t))+\frac{\beta}{4}) \\
\nonumber &\qquad \qquad \ge \cos(\frac{\pi}{2} - \frac{\alpha}{2} + \frac{\beta}{4})=\sin(\frac{\alpha}{2} - \frac{\beta}{4}),
\end{align}
for $ k \in \{1,2, \cdots, M\}$.
Using   \eqref{ss1}-\eqref{ss2}, we have that,
 for $t - \tau \le s \le t$, 
\begin{align}\label{fn}
F_N(s)  &= \frac{1}{N}\sum_{k=1}^N\Big( \sin (\theta_k(s) - \theta_i(s)) - \sin (\theta_k(s) - \theta_j(s))   \Big)\\
\nonumber& =-\frac{2}{N} \sin(\frac{\theta_i(s) - \theta_j(s)}{2})\Big( \sum_{k =1}^M \cos(\theta_k(s) - \frac{\theta_i(s) + \theta_j(s)}{2})
\\
\nonumber& \qquad+  \sum_{k =M+1}^N \cos(\theta_k(s) - \frac{\theta_i(s) + \theta_j(s)}{2})\Big)\\
\nonumber&\le -\frac{2}{N} \Big( M\sin(\frac{\alpha}{2} - \frac{\beta}{4}) \cos(\frac{\alpha}{2}+\frac{5}{8}\beta)-(N-M)\cos(\frac{\alpha}{2} - \frac{\beta}{8})\Big)
 \Big)\\
\nonumber &\le -2\mu.
\end{align}
Here we have used an elementary trigonometric identity for the second equality, i.e.,
\begin{multline*}
 \sin (\theta_k(s) - \theta_i(s)) - \sin (\theta_k(s) - \theta_j(s))  
 \\
  =- 2 \sin(\frac{\theta_i(s) - \theta_j(s)}{2}) \cos(\theta_k(s) - \frac{\theta_i(s) + \theta_j(s)}{2}).
 \end{multline*} 
Using \eqref{fn} together with the fact that $|F_N(s)|\le2$ for all $s\ge0$, we have
\begin{align*}
\int_0^t \frac{1}{m} e^{-\frac{1}{m}(t -s)} F_N(s) ds &=  \int_{t -\tau}^t \frac{1}{m} e^{-\frac{1}{m}(t -s)} F_N(s) ds + \int_0^{t -\tau} \frac{1}{m} e^{-\frac{1}{m}(t -s)} F_N(s) ds\\
&\le  -2\mu(1 -e^{-\frac{1}{m}\tau})
+2e^{-\frac{1}{m}\tau}.
\end{align*}
Making use of this for \eqref{difference0}, we obtain
\begin{align}
\label{decay}\dot{\theta}_i(t) -  \dot{\theta}_j(t)&=(\dot{\theta}_i(0) - \dot{\theta}_j(0)) e^{-\frac{1}{m}
t} + (\omega_i - \omega_j)(1-
e^{-\frac{1}{m}t}) \\
\nonumber & + K\int_0^t \frac{1}{m} e^{-\frac{1}{m}(t -s)} F_N(s) ds\\
\nonumber&\le K\Big( \lambda e^{-\frac{1}{m}t} + \mu (1 - e^{-\frac{1}{m}t})    -2\mu(1 -e^{-\frac{1}{m}\tau})
+2e^{-\frac{1}{m}\tau} \Big)\\
\nonumber &\le K(\lambda e^{-\frac{1}{m}\tau}+\mu-2\mu(1 -e^{-\frac{1}{m}\tau})
+2e^{-\frac{1}{m}\tau}  ) \\
\nonumber &\le K(-\mu+(\lambda + 2 \mu +2)e^{-\frac{1}{m} \tau})\\
\nonumber&<0,
\end{align}
where we have used \eqref{tau2} and \eqref{mk2} for the derivation of the last inequality.
By \eqref{decay}, we have 
$$
\sup_{t \ge 0} D(\Theta_M(t)) \le \pi -\alpha.
$$
Moreover, we have
$$
D(\Theta_M(t)) \le \pi- \alpha-\beta
$$
for all
$$
t > \tau + \frac{\beta}{K(\mu-(\lambda + 2 \mu +2)e^{-\frac{1}{m} \tau})}.
$$
This completes the proof.
\qed

\begin{remark}
\label{rm1}
Assume the conditions \eqref{mu2}--\eqref{domega2} hold and $\Theta(t)$ is a solution of \eqref{kuramoto2nd}. 
Suppose at certain moment $t=t_0$ and for some $\theta_* \in \mathbb R$, the $M$ oscillator $\Theta_M(t_0)$ happen to be located in a sector \begin{equation}
S:= \Big\{  \theta + 2k \pi :  \big| \theta - \theta_* \big| \le \frac{1}{2}( \pi - \alpha-\beta)  \text{ and }  k \in \mathbb Z\Big\}\\
\end{equation}
of argument $\pi - \alpha - \beta$ and 
$$
 D(\dot{\Theta}_{M}(t_0)) < \lambda K.
 $$

 One can choose appropriate   $\{k_i\}_{i=1}^M \subset \mathbb{Z}$,   $\bar{c} \in \mathbb R$, and $\bar{\theta}_i \in [0, \pi-\alpha-\beta]$  so that
\begin{equation}
\theta_i(t_0) = \bar{\theta}_i+\bar{c}+2k_i\pi,
\end{equation}
for $i=1,2,3,\cdots,M.$ Now, let $\hat{\Theta}(t)$ be the solution to \eqref{kuramoto1}  with the initial conditions
\begin{equation}
 \hat{\theta}_i(0)= \left\{ \begin{aligned}
 & \bar{\theta}_i   &&\text{ for } i=1,2,\cdots,M, \\
 & \theta_i(t_0)-\bar{c} &&\text{ for } i=M+1,M+2,\cdots,N, \end{aligned}
 \right.  
\end{equation}
and
\begin{equation}
 \dot{\hat{\theta}}_i(0)= \dot{\theta}_i(t_0) \qquad \text{ for } i=1,2,3\cdots,N.  
\end{equation}
By the uniqueness of the solution to  \eqref{kuramoto1}, we see that
 for all $t \ge t_0$
 \begin{equation}
 \theta_i(t) = \left\{
 \begin{aligned}
 & \hat{\theta}_i(t-t_0)+\bar{c}+2k_i\pi &&\text{ for } i=1,2,3,\cdots,M,\\
 & \hat{\theta}_i(t-t_0)+\bar{c}  && \text{ for } i=M+1,M+2,\cdots,N. 
  \end{aligned}
  \right.
 \end{equation}
 On the other hand, we see that $\hat{\Theta}(t)$ satisfies the conditions \eqref{alpha2} - \eqref{lambda2}. Hence, by Lemma~\ref{lem3}, we obtain that
 \begin{equation*}
D(\hat{\Theta}_M(t))\le \pi - \alpha
\end{equation*}
 for all $t \ge 0$, and
 \begin{equation}
 D(\hat{\Theta}_M(t))\le  \pi -\alpha -\beta
  \end{equation}
 for all 
 $$
t > \tau + \frac{\beta}{K(\mu-(\lambda + 2 \mu +2)e^{-\frac{1}{m} \tau})},
$$
 which implies that
\begin{enumerate}
\item[(a)] $\Theta_M(t)$ will remain in a moving sector of argument $ \pi-\alpha-\beta $ in the phase space $S^1$ for all 
$$
t > t_0+ \tau + \frac{\beta}{K(\mu-(\lambda + 2 \mu +2)e^{-\frac{1}{m} \tau})}.$$
\item[(b)] $D(\Theta_M(t))$ is uniformly bounded for all $t \ge 0.$
\end{enumerate}
 
\end{remark}

\begin{lemma}
\label{lemma8}
Under the assumptions \eqref{mu2}- \eqref{domega2}, suppose $\Theta(t)$ is a solution of \eqref{kuramoto2nd} that satisfies \eqref{alpha2}-\eqref{lambda2}. Then we have
\begin{equation}
\sup_{t \ge 0} D(\Theta(t)) < \infty.
\end{equation}
\end{lemma}
\begin{proof}
By \eqref{alpha2} and in view of Remark~\ref{rm1}, without loss of generality, we may  assume 
\begin{equation*}
\begin{aligned}
0 \le \theta_i(0) &\le \pi-\alpha-\beta  &&{for }\, i=1,2,3\cdots,M\\
0 \le \theta_i(0) &\le 2\pi &&{for }\, i=M+1,M+2,M+3\cdots,N.
\end{aligned}
\end{equation*}  
By Lemma~\ref{lem3}, we see that there exists a moving sector $S(t)=[s(t), s(t)+\pi-\alpha-\beta]$ of argument $\pi-\alpha-\beta$ that   covers all the oscillators of $\Theta_M(t)$ for all $$t > t_0:=\tau + \frac{\beta}{K(\mu-(\lambda + 2 \mu +2)e^{-\frac{1}{m} \tau})},$$ where $s(t)$ is the continuous function defined  by
$$
s(t) := \min\{ \theta_1(t), \theta_2(t), \cdots, \theta_M(t) \}.
$$
 Applying Lemma~\ref{meanphase} and Lemma~\ref{lem3}, we see that if $D(\Theta(t))$ is not  bounded, there exists at least one oscillator  $\theta_k(t) \in \{ \theta_{M+1}(t), \theta_{M+2}(t), \cdots, \theta_N(t)\}$ such that  $\theta_k(t) - s(t)$ is an unbounded continuous function.   
Without loss of generality, we may assume $k=M+1$ and  there exists a strictly increasing sequence $\{t_j\}_{j=1}^{\infty}$ with $t_1>t_0$ 
and
\begin{equation}
\label{tinfinity}
\lim_{j \to \infty} t_j = \infty
\end{equation}
such that
\begin{equation}
\label{increase}
\theta_{M+1}(t_j) - s(t_j) = \theta_{M+1}(t_0) - s(t_0)+2j\pi,
\end{equation}
for $j=1,2,3,\cdots.$
By Lemma~\ref{lemma1}, \eqref{tinfinity} and the fact $\lambda > \mu +2$, we see that there exists a positive integer $j_*$ such that 
for all $t \ge t_{j_*}$, we have
\begin{equation}
\label{speeddifferencefinite}
\begin{aligned}
| \dot{\theta}_i(t) -  \dot{\theta}_j(t) |& \le | \dot{\theta}_i(0) - \dot{\theta}_j(0)|  e^{-\frac{1}{m}
t}+ (| \omega_i - \omega_j |+2K)(1 - e^{-\frac{1}{m} t})\\
& < \lambda K,
\end{aligned}
\end{equation}
for all $i,j =1,2,3, \cdots,N.$
By the continuity of $\theta_{M+1}(t) - s(t)$, we see that there exists $t_{*} \in [t_{j_*}, t_{j_*+1}]$ such that 
\begin{equation*}
\theta_{M+1}(t_*) - s(t_*) =2m\pi
\end{equation*}
for some integer $m$. 
 Since $D(\Omega_{M+1}) \le D(\Omega) < \mu K$ and the condition \eqref{mu3} holds by replacing $M$ by $M+1$, taking \eqref{speeddifferencefinite} into account, applying Lemma~\ref{lem3}(c.f. Remark~\ref{rm1}), we see that $$\theta_{M+1}(t) \in [s(t)+2m\pi-(\pi-\alpha-\beta), s(t)+2m\pi+(\pi-\alpha-\beta)]$$
 for all 
$
t > t_0 + t_*,
$
which violates \eqref{increase}. We therefore conclude that
$$
\sup_{t \ge 0}D(\Theta(t)) < \infty.
$$
\end{proof}

Now we use the previous Lemmas to prove  Theorem~\ref{thm2} - \ref{thm3}: 

{\sc Proof of Theorem~\ref{thm2}}
 The result is a direct consequence of Lemma~\ref{lemdiameter}, Lemma~\ref{lem3} and Lemma~\ref{lemma8}.
  
\qed
 
{\sc Proof of Theorem~\ref{thm3}.}\\
\noindent Suppose $\Theta(t)$ is a solution that satisfies all the conditions described in Theorem~\ref{thm3}.
Applying Theorem~\ref{thm2}, we see that
\begin{equation}
\label{dotzero}
\lim_{t \to \infty} \dot{\Theta}(t) = 0.
\end{equation}
On the other hand, by taking derivatives on \eqref{kuramoto1} and applying Lemma~\ref{lemmathetai}   and Lemma~\ref{lemmathetai2}, we obtain the uniform boundedness of all third derivatives of $\Theta(t)$. This together with 
\eqref{dotzero} and Lemma~\ref{cvl2} in the appendix gives
\begin{equation}
\label{ddotzero}
\lim_{t \to \infty} \ddot{\Theta}(t) = 0.
\end{equation}
Applying Lemma~\ref{lem3}, we see that 
\begin{equation}
\label{bibibamp}
D(\Theta(t) ) \le \pi- \alpha-\beta < \pi
\end{equation}
for all 
$$
t > \tau + \frac{\beta}{K(\mu-(\lambda + 2 \mu +2)e^{-\frac{1}{m} \tau})},
$$
Plugging \eqref{dotzero} and \eqref{ddotzero} into the limiting equation of \eqref{kuramoto1} and taking \eqref{bibibamp} into account, we see that
\begin{equation}
\lim_{t \to \infty}\sin (\theta_i(t) - \theta_j(t)) =0,
\end{equation}
for $i,j \in \{1,2,3, \cdots, N\}$, and hence
\begin{equation}
\lim_{t \to \infty} (\theta_i(t) - \theta_j(t)) =0,
\end{equation}
for $i,j \in \{1,2,3, \cdots, N\}$.
 This completes the proof of Theorem~\ref{thm3}. 

\qed

\section{Unconditional frequency synchronization} 
This section is devoted to the proof of Theorem~\ref{theorem3}.
To prove Theorem~\ref{theorem3}, it is sufficient to  show that for any given initial conditions $\Theta(0)$ and $\dot{\Theta}(0),$
the diameter function  $D(\Theta(t))$ of the solution $\Theta(t)$ of \eqref{kuramoto1} is bounded for all $t > 0,$i.e.,
$$
\sup_{t \ge 0} D(\Theta(t)) < \infty.
$$ 

First, we choose a large positive number $T_1$ such that 
\begin{equation}
D(\dot{\Theta}(0)) e^{-\frac{1}{m}T_1} < \frac{1}{10} \sin \frac{\alpha_3}{16} \cos \frac{11 \alpha_3}{16} K.
\end{equation}
Hence, by  Lemma~\ref{lemma1} together with \eqref{beta3} and \eqref{domega3}, we have
\begin{equation}
\label{speed1}
D(\dot{\Theta}(t)) \le \big(\sin \frac{\alpha_3}{16} \cos \frac{11 \alpha_3}{16}+2\big) K
\end{equation}
for all $t \ge T_1.$
We see also from Lemma~\ref{lemma1} that
\begin{equation}
\label{speed2}
D(\dot{\Theta}(t)) \le D(\dot{\Theta}(0)) + D(\Omega)+2K
\end{equation}
for all $t \ge 0.$
 We then set
\begin{equation}
\label{gamma1}
\gamma = \frac{\alpha_3}{4  \big(\sin \frac{\alpha_3}{16} \cos \frac{11 \alpha_3}{16}+2\big) K }.
\end{equation}
From now on, we fix a  positive integer $n$ that satisfies
\begin{equation}
\label{nbound}
\frac{4n\pi + \alpha_3 - D(\Theta(0))}{D(\dot{\Theta}(0)) + D(\Omega)+2K} > T_1 +2\gamma.
\end{equation}
{
If there exists a moment such that $D(\Theta(t)) \ge 4n\pi + \alpha_3$, by \eqref{speed2} and \eqref{nbound}, the first moment
that $D(\Theta(t))$ hits  $4n\pi + \alpha_3$ is greater than $T_1+\gamma.$ 
}Let $t_0 > T_1+\gamma$ be the first moment
that $D(\Theta(t))$ hits  $4n\pi + \alpha_3,$ and assume that $\theta_{i_1}(t_0) - \theta_{i_2}(t_0)$ is one of the representations of $D(\Theta(t_0))$, i.e., 
 \begin{equation}
\label{diameterm}
D(\Theta(t_0))= \theta_{i_1}(t_0) - \theta_{i_2}(t_0) = 4n\pi + \alpha_3.
\end{equation} 
By \eqref{speed1} and \eqref{gamma1}, we have, for $t_0 - \gamma \le s\le t_0$, 
\begin{align*}
&4n\pi + \frac{3}{4}\alpha_3 \le \theta_{i_1}(s) - \theta_{i_2}(s) \le 4n\pi + \frac{5}{4}\alpha_3,\\
&2n\pi + \frac{3}{8}\alpha_3 \le \frac{\theta_{i_1}(s) - \theta_{i_2}(s)}{2} \le 2n\pi + \frac{5}{8}\alpha_3.
\end{align*}
Since $\alpha_3 = \frac{2\pi}{5},$ we have 
\begin{align*}
& 0<\sin\frac{3 \alpha_3}{4} \le \sin(\theta_{i_1}(s) - \theta_{i_2}(s)) \le \sin \frac{5 \alpha_3}{4},\\
& 0<\sin\frac{3 \alpha_3}{8} \le \sin(\frac{\theta_{i_1}(s) - \theta_{i_2}(s)}{2}) \le \sin \frac{5 \alpha_3}{8},
\end{align*}
and hence, for $t_0 - \gamma \le s\le t_0,$
\begin{align}
\label{fs3}F_3(s)  & = \frac{1}{3}\Big(  -2 \sin (\theta_{i_1}(s) - \theta_{i_2}(s)) \\
\nonumber&\,\,\,-2 \sin(\frac{\theta_{i_1} (s)- \theta_{i_2}(s)}{2})\cos(\theta_{i_3}(s) - \frac{\theta_{i_1}(s) + \theta_{i_2}(s)}{2}) \Big)\\
\nonumber &\le \frac{1}{3}\Big( -2 \sin \frac{3\alpha_3}{4} + 2 \sin \frac{5\alpha_3}{8} \Big) \\
\nonumber & \le- \frac{4}{3} \sin \frac{\alpha_3}{16} \cos \frac{11 \alpha_3}{16}.
\end{align}
By \eqref{difference0}, we have
 \begin{align}
 \label{bound3}\dot{\theta}_{i_1}(t_0)& -  \dot{\theta}_{i_2}(t_0) = (\dot{\theta}_{i_1}(0) - \dot{\theta}_{i_2}(0))e^{-\frac{1}{m}
t_0}+(\omega_i - \omega_j) (1-
e^{-\frac{1}{m}t_0}) \\
\nonumber&+\int_{t_0-\gamma}^{t_0} \frac{K}{m} e^{-\frac{1}{m}(t_0 -s)} F_3(s) ds + \int_{0}^{t_0-\gamma} \frac{K}{m} e^{-\frac{1}{m}(t_0 -s)} F_3(s) ds\\
\nonumber&\le K\Big( \frac{1}{10}\sin \frac{\alpha_3}{16} \cos \frac{11 \alpha_3}{16} +\frac{9}{10}\sin \frac{\alpha_3}{16} \cos \frac{11 \alpha_3}{16}  \\
\nonumber&\,\,  -\frac{4}{3}\sin \frac{\alpha_3}{16} \cos \frac{11 \alpha_3}{16}(1 -e^{-\frac{\gamma}{m}}) + 2e^{-\frac{\gamma}{m}} \Big) \\
\nonumber&\le K\Big((2 + \frac{4}{3}\sin \frac{\alpha_3}{16} \cos \frac{11 \alpha_3}{16})e^{-\frac{\gamma}{m}} -\frac{1}{3}\sin \frac{\alpha_3}{16} \cos \frac{11 \alpha_3}{16}\Big)\\
\nonumber& < 0,
\end{align}
where we have employed \eqref{mk3} and \eqref{gamma1} in the derivation of the last inequality.
Since \eqref{bound3} holds true for any choice of the representation for the diameter function $D(\Theta(t))$, we conclude  
that
\begin{equation}
D(\Theta(t)) \le 4n\pi+\alpha_3
\end{equation}
for all $t > T_1 + \gamma.$ This gives the global boundedness of $D(\Theta(t))$. Thanks to Lemma~\ref{lemdiameter}, the proof of Theorem~\ref{theorem3} is complete.  
\qed



\section{Appendix}\label{app}
In this appendix, we present some elementary lemmas that are used in our main theorems.
\begin{lemma}\label{cvl}
Let $f\in C(0,\infty)$ be uniformly continuous and $f\ge0$ on $(0,\infty)$. Assume further 
$\int_0^\infty f(s) ds <\infty$. Then we have that
$f(s)\to 0 \text{ as } s\to\infty.$
\end{lemma}
\begin{proof}
Suppose that the statement is not true. Then, for some $\varepsilon_0>0$, there exists a sequence $\{t_n\}_{1}^\infty\subset [1,\infty)$ such that $f(t_n)\ge \varepsilon_0$ and $t_{n+1} -t_n\ge 1$.
Since $f$ is uniformly continuous, there exists $\delta_0\in(0,1/2)$ such that $|f(t) -f(s)|<\varepsilon_0/2$ whenever $|t-s|<\delta_0$. Now it is easy to see that $$\int_{t_n-\delta_0}^{t_n+\delta_0} f(t) dt \ge \int_{t_n-\delta_0}^{t_n+\delta_0} f(t_n) dt - \int_{t_n-\delta_0}^{t_n+\delta_0} |f(t) - f(t_n) | dt \ge \delta_0 \varepsilon_0>0.$$
This implies that 
$\int_0^\infty f(t) dt\ge \sum_{n=1}^\infty \int_{t_n-\delta_0}^{t_n+\delta_0} f(t) dt =\infty$, which is a contradiction. We are done.
\end{proof}

\begin{lemma}\label{cvl2}
Let $f\in C^2(0,\infty)$. Suppose that  $\lim_{t\to\infty} f(t) =0$
and $\sup_{t\ge0} |f'(t)| + |f''(t) |\le M$ for some $M>0$. Then $\lim_{t\to\infty} f'(t) =0$.
\end{lemma}
\begin{proof}
Suppose that the statement is not true. Then, for some $\varepsilon_0>0$, there exists a sequence $\{t_n\}_{1}^\infty\subset [1,\infty)$ such that $|f' (t_n)|\ge \varepsilon_0$ and $t_{n+1} -t_n\ge 1$. Without loss of generality, we may assume that $f'(t_n) \ge \varepsilon_0$. 
By the condition $|f''(t) | \le M$, we have 
$$f'(t_n + \tau ) \ge \frac{\varepsilon_0}{2}$$
if $|\tau|\le \frac{\varepsilon_0}{2M}$. Then we have
$$f(t_n + \frac{\varepsilon_0}{2M}) - f(t_n - \frac{\varepsilon_0}{2M}) = \int_{-\frac{\varepsilon_0}{2M}}^{\frac{\varepsilon_0}{2M}} f'(t_n + \tau ) d\tau \ge \frac{\varepsilon_0^2}{2M} >0.  $$
This contradicts to $\lim_{t \to \infty}f(t)=0$. The proof is completed.
\end{proof}

  \section*{Acknowledgments.}
C.-Y. J. was supported by the Basic Science Research Program through the National Research Foundation of Korea funded by the Ministry of Education(2015R1D1A1A01059837). B. K. was supported by Basic Science Research Program through the National Research Foundation of Korea(NRF) funded by the Ministry of science, ICT and future planning(2015R1C1A1A02037662).


\end{document}